\documentclass{amsart}

\usepackage{amsmath}
\usepackage{amsxtra}
\usepackage{amscd}
\usepackage{amsthm}
\usepackage{amsfonts}
\usepackage{amssymb}
\usepackage{mathtools}
\usepackage{eucal}
\usepackage[all]{xy}
\usepackage{graphicx}
\usepackage{mathrsfs}
\usepackage{color}
\usepackage{caption}
\usepackage{enumerate}
\usepackage{enumitem}
\usepackage{faktor}
\usepackage{bbm}
\usepackage{hyperref}
\usepackage[T1]{fontenc} 
\usepackage[utf8]{inputenc} 
\usepackage{tikz}
\usepackage{tikz-cd}
\usetikzlibrary{matrix,arrows,decorations.pathmorphing}

\usepackage{color,soul}

\usepackage{changepage}
\usetikzlibrary{automata}
\usetikzlibrary{positioning} 
\usetikzlibrary{arrows} 
\usetikzlibrary{arrows.meta}
\tikzset{
  node distance=2.5cm,
  every state/.style={
    semithick,
    fill=gray!10},
  initial text={},
  double distance=2pt,
  every edge/.style={
    draw,
    ->,>=stealth,
    auto,
    semithick},
  every loop/.style={
    >=Latex, 
    semithick 
  }
}
\usepackage{xcolor}
\definecolor{gray}{RGB}{180,180,180}
\definecolor{red}{RGB}{220,70,70}
\definecolor{yellow}{RGB}{240,220,70}
\definecolor{green}{RGB}{70,180,70}
\definecolor{blue}{RGB}{70,120,220}

\hypersetup{
  colorlinks = true,
  urlcolor = cyan,
}

\theoremstyle{plain}
\newtheorem{theorem}{Theorem}
\newtheorem{lemma}[theorem]{Lemma}
\newtheorem{proposition}[theorem]{Proposition}
\newtheorem{corollary}[theorem]{Corollary}

\theoremstyle{definition}
\newtheorem{definition}[theorem]{Definition}
\newtheorem{example}[theorem]{Example}
\newtheorem{nonexample}[theorem]{Nonexample}
\newtheorem{conjecture}[theorem]{Conjecture}

\theoremstyle{remark}
\newtheorem{remark}[theorem]{Remark}

\numberwithin{equation}{section}

\newcommand{\F}{\mathbb{F}}

\newcommand{\N}{\mathbb{N}}

\renewcommand{\P}{\mathbb{P}}
\newcommand{\Q}{\mathbb{Q}}

\newcommand{\Z}{\mathbb{Z}}

\newcommand{\calo}{\mathcal{O}}
\newcommand{\calp}{\mathcal{P}}

\newcommand{\mfm}{\mathfrak{m}}

\newcommand{\mfp}{\mathfrak{p}}

\newcommand{\rmn}{\mathrm{N}}

\newcommand{\Per}{\mathrm{Per}}
\newcommand{\Preper}{\mathrm{PrePer}}

\newcommand{\lcm}{\mathrm{lcm}}
\newcommand{\lra}{\longrightarrow}

\newcommand{\paren}[1]{\mathopen{}\left(#1\right)\mathclose{}}
\newcommand{\set}[1]{\mathopen{}\left\{#1\right\}\mathclose{}}

\newcommand{\verts}[1]{\mathopen{}\left\lvert#1\right\rvert\mathclose{}}

\newcommand\restr[2]{{\left.\kern-\nulldelimiterspace#1\right|_{#2}}}
\renewcommand{\setminus}{-}

\makeatletter
\AtBeginDocument
 {
    \def\@thm#1#2#3{%
      \ifhmode
        \unskip\unskip\par
      \fi
      \normalfont
      \trivlist
      \let\thmheadnl\relax
      \let\thm@swap\@gobble
      \let\thm@indent\indent 
      \thm@headfont{\scshape}
      \thm@notefont{\fontseries\mddefault\upshape}%
      \thm@headpunct{.}
      \thm@headsep 5\p@ plus\p@ minus\p@\relax
      \thm@space@setup
      #1
      \@topsep \thm@preskip               
      \@topsepadd \thm@postskip           
      \def\dth@counter{#2}%
      \ifx\@empty\dth@counter
        \def\@tempa{%
          \@oparg{\@begintheorem{#3}{}}[]%
        }%
      \else
        \H@refstepcounter{#2}%
        \hyper@makecurrent{#2}%
        \let\Hy@dth@currentHref\@currentHref
        \AddToHookNext{para/begin}{\MakeLinkTarget*{\Hy@dth@currentHref}}%
        \def\@tempa{%
          \@oparg{\@begintheorem{#3}{\csname the#2\endcsname}}[]%
        }%
      \fi
      \@tempa
    }%
  \dth@everypar={%
    \@minipagefalse \global\@newlistfalse
    \@noparitemfalse
    \if@inlabel
      \global\@inlabelfalse
      \begingroup \setbox\z@\lastbox
       \ifvoid\z@ \kern-\itemindent \fi
      \endgroup
      \unhbox\@labels
    \fi
    \if@nobreak \@nobreakfalse \clubpenalty\@M
    \else \clubpenalty\@clubpenalty \everypar{}%
    \fi
  }%
}

\title[Periodic points of polynomials with good reduction]{Uniform bounds on periodic points of polynomials with good reduction}

\author{Isaac Rajagopal}
\address{Department of Mathematics, Massachusetts Institute of Technology}
\email{isaacraj@mit.edu}

\author{Robin Zhang}
\address{Department of Mathematics, Massachusetts Institute of Technology}
\email{robinz@mit.edu, robinzmit@gmail.com}

\date{May 4, 2026}

\begin{document}

\begin{abstract}
    We establish effective bounds
    on the number of periodic
    points of polynomials $\phi$
    defined over $p$-adic fields and number fields,
    under a mild reduction hypothesis
    that is satisfied by all unicritical polynomials
    $X^d + c$ with $c$ integral at some prime dividing $d$.
    As a consequence,
    we verify the uniform boundedness conjecture for
    this class of polynomials over number fields $K$,
    giving the explicit uniform bound $\#\mathrm{Per}_K(\phi) \leq d^{[K:\mathbb{Q}]}$.
\end{abstract}

\maketitle

\setcounter{tocdepth}{1}
\tableofcontents


\section{Introduction}
\label{sec:intro}
\subsection{Uniform boundedness of periodic points}
Given a set $S$ and
a function $\phi: S \lra S$,
there is a discrete dynamical system
given by the discrete iterations of $\phi$ on $S$.
Let $\phi^n$ denote the $n$-fold composition
$\phi \circ \phi \circ \cdots \circ \phi: S \lra S$.
One of the central problems in arithmetic dynamics is
determining the periodic points of $\phi$ when $S$
has some integral or rational structure.
\begin{definition}
    Define the set of \textit{periodic points} of $\phi$ in $S$ to be
    \[
        \Per_S(\phi) \coloneq \set{x \in S
        \mid \phi^n(x) = x \text{ for some } n \in \N}.
    \]
    We say that $x \in \Per_S(\phi)$ has \emph{period} $n$
    if $\phi^n(x) = x$,
    and say that $x$ has \emph{exact period} $n$
    if $n$ is the minimal positive integer such that $\phi^n(x) = x$.
\end{definition}

The dynamical uniform boundedness conjecture
of Morton and Silverman \cite{Morton94}
predicts that for endomorphisms $\phi$
on projective space over a number field $K$,
there is an upper bound on the number of periodic points
that depends only on the dimension $N$ of the space,
the degree $d$ of $\phi$, and
the degree $D \coloneq  [K:\Q]$ of the field extension.
Note that uniform boundedness for periodic
points is equivalent to uniform boundedness
for preperiodic points in many cases, e.g.
for unicritical polynomials $\phi_{d, c}(X) \coloneq  X^d + c$
by \cite{DP20}; we will mostly focus on periodic points
in this article.

\begin{conjecture}[Uniform boundedness conjecture for periodic points]
    \label{conj:uniform-boundedness}
    Fix integers $N \geq 1$, $d \geq 2$, and $D \geq 1$.
    There exists a constant $C(N, d, D) \in \N$
    such that for all number fields $K/\Q$ of degree $D$
    and all degree-$d$ morphisms $\phi: \P^N \lra \P^N$ defined over $K$,
    \[
        \# \Per_{\P^N(K)}(\phi) \leq C(N,d,D)\;.
    \]
\end{conjecture}

Assuming one of Vojta's higher-dimensional
generalizations of the $abc$ conjecture
(see \cite[Conjecture 2.1]{looper-2021},
\cite[Conjecture 2.9]{zhang-2024}),
Looper \cite{looper-2021,looper-2025} conditionally proved
a weak form
(in which $C$ depends on $K$)
of Conjecture \ref{conj:uniform-boundedness}
for polynomials $\phi \in K[X]$.
However,
Conjecture~\ref{conj:uniform-boundedness} is not known
unconditionally for any tuple $(N, d, D)$,
even when $(N, d, D) = (1, 2, 1)$.
One goal of this paper is to
establish uniform bounds
in the $N=1$ case
for a class of polynomials
that are unconditional and explicit.
Since endomorphisms of $\P^1$ defined over $K$ are
defined only up to scaling by constants in $K^\times$,
we may (and do) normalize $\phi$
by multiplying by a constant so that
it is monic.

As an endomorphism of $\P^1(K)$,
we say that a polynomial
$\phi = \sum_{i=0}^d a_i X^i \in K[X]$ has \emph{good reduction}
at a prime $\mfp$ of $\calo_K$
if each $a_i$ is integral at $\mfp$
and the leading coefficient $a_d$
is a $\mfp$-unit,
i.e. $v_{\mfp}(a_i) \geq 0$ for all $i$
and $v_\mfp(a_d) = 0$.
We say that $\phi$
has \emph{inseparable good reduction}
if $\phi$ has good reduction and furthermore
$a_i \in \mfp$ for each $i$ coprime to $\mfp$.
For polynomials with
inseparable good reduction
at a prime $\mfp$,
we obtain a bound on periodic points
that depends only on the norm of $\mfp$;
 this yields
a simple uniform bound
$\# \Per_K(\phi) \leq d^D$
because $\mfp \mid d$.

\begin{theorem}
    \label{thm:uniform-boundedness}
    Fix integers $d \geq 2$ and $D \geq 1$.
    Let $K/\Q$ be a number field of degree $D$,
    $\mfp$ be a prime of $\calo_K$,
    and $\phi = \sum_{i=0}^d a_i X^i \in K[X]$
    be a polynomial with inseparable good reduction at $\mfp$.
    Then there is an upper bound in terms of
    the ideal norm of $\mfp$,
    \[
        \# \Per_K(\phi) \leq \rmn_{K/\Q}(\mfp),
    \]
     which in particular is uniformly bounded above by $d^D$.
\end{theorem}

\begin{remark}
    \label{remark:local-comparison}
    Unlike earlier effective bounds on the number of periodic points
    which are given in terms of $d$, $D$, and
    the number $s$ of primes of bad reduction
    by \cite{benedetto-2007,Canci07,Canci10,CP16,CTV19,CV19,Morton94,Narkiewicz89,Troncoso17},
    our hypothesis assumes a single place of good reduction
    and yields an explicit uniform bound on $\#\Per_K(\phi)$
    that is independent of $s$.
    Examples of such bounds are
    \begin{itemize}
        \item Benedetto \cite{benedetto-2007}: $\#\Per_K(\phi) \leq O(s\log (s))$
            where the $O$ constant depends only on $d$ and $D$, and
        \item Canci--Vishkautsan \cite{CV19}:
            $\#\Per_K(\phi) \leq 2^{2^5s}d+2^{2^{77}s}$.
    \end{itemize}
    We also note an intriguing connection to
    the classical picture for abelian varieties $A$
    (see \cite[Theorem C.1.4]{hindry-silverman}):
    if an abelian variety $A$ has good reduction at
    a non-archimedean place $v$ of residue characteristic $p$,
    then the reduction map is injective on $A[n]$
    for every integer $n$ coprime to $p$.
    By contrast, the ``good reduction $\Rightarrow$ injectivity
    of reduction on torsion'' phenomenon
    that we recover in Theorem \ref{thm:uniform-boundedness}
    occurs only in the special situation $p \mid d$.
    This helps explain why the mild local condition in
    Theorem \ref{thm:uniform-boundedness}
    can already force strong uniformity.
\end{remark}

Note that a polynomial $\phi$ with good reduction at some prime $\mfp$
automatically has inseparable good reduction at $\mfp$
if it is postcritically bounded (\cite[Lemma 8]{feiner})
or if it is of the form $\phi(X^k)$ for $k$ divisible by $\mfp$.
For instance, Theorem \ref{thm:uniform-boundedness}
implies that the
upper bound $\#\Per_K(\phi) \leq d^D$ applies uniformly
to all unicritical polynomials $\phi_{d, c}$ with
$c \in  \bigcup_{\mfp \mid d} \, \calo_{K, \mfp}$,
where
\[
    \calo_{K, \mfp} \coloneq  \set{x \in K \mid v_\mfp(x) \geq 0}.
\]
The unicritical $(d, D) = (2, 1)$ case of
Theorem \ref{thm:uniform-boundedness}
recovers results of Walde--Russo
\cite[Corollaries 6 and 7]{walde-russo}.
The unicritical case of
Theorem \ref{thm:uniform-boundedness} also
implies part of the
recent $S$-integer uniform boundedness results
of Doyle--Hindes \cite[Corollary 1.2]{DH25} for fixed degree $d$, 
when $S$ does not contain
all of the prime factors of $d$.

The proof of Theorem \ref{thm:uniform-boundedness}
is based on showing that periodic points occupy
distinct residue classes in $\calo_K /\mfp \simeq \F_{p^f}$
for $f$ the residue degree of $\mfp$;
this separation is achieved by a $p$-adic contraction
(resp. expansion) argument
showing that points that are close to (resp. far from)
a periodic point of $\phi_{d, c}$ are
iteratively attracted to
(resp. repulsed from) their orbits; see Figure \ref{figurediagram} for a picture of this. 

\subsection{Applications to exact periods}
For polynomials of prime power degree,
we use additional $p$-adic analysis to
improve the upper bound on exact periods of $\phi$
beyond Theorem \ref{thm:uniform-boundedness}.
For ideals $A, B, C, D \subset \calo_K$,
define $m(A,B,C,D)$ to be the smallest positive integer $m$ such that $A \mid C \cdot m$ and $B \mid D \cdot m$.
For principal ideals generated by integers $a, b, c, d$,
observe that $m(a, b, c, d) 
= \lcm\paren{\frac{a}{\gcd(a, c)}, \frac{b}{\gcd(b, d)}}$.
\begin{theorem}
\label{thm:p-power-period}
    Let $K$ be a number field, $d = p^k$ be a prime power, and
    $\mfp$ be a prime of $\calo_K$ above $p$
    with residue degree $f$. Let $\phi = \sum_{i=0}^d a_i X^i \in K[X]$ be a degree-$d$ polynomial with good reduction at $\mfp$ such that $a_i \in \mfp$
    if $i \not \in \{0,d\}$. 
    If $f \mid k$ or $a_0 \in  \Q$,
    then each $x \in \Per_K(\phi)$ has period $m(\mfp, f, a_0, k)$.
\end{theorem}
\begin{example}[see Corollary \ref{cor:2-general}]
    Let $K$ be a number field,
    $f$ be the residue degree of a
    prime of $\calo_K$ above $2$,
    and $c =\frac{r}{s} \in \Q$ with $s$ odd.
    Theorem \ref{thm:uniform-boundedness}
    implies that there are at most $2^f$ periodic points 
    of $\phi(X) \coloneq  X^2 + c$ in $K$.
    Furthermore, Theorem~\ref{thm:p-power-period} shows
    that all of the periodic points
    have period $f$ (resp. $2f$) when
    $rf$ is even (resp. $rf$ is odd).
\end{example}
Notice that Theorem \ref{thm:p-power-period}
yields an upper bound
$m(\mfp, f, a_0, k) \leq \lcm(p,f) \leq pf$
on exact periods that is stronger than
the general good reduction bounds
due to \cite{Morton94,Narkiewicz89,Pezda942,Pezda94,Zieve96}
(cf. \cite[Section 2.6]{Silverman07}),
such as:
\begin{itemize}
    \item Morton--Silverman \cite[Corollary B]{Morton94}:
    if $\phi$ has good reduction at two primes $\mfp_1$ and $\mfp_2$ over $p_1$ and $p_2$
    with residue degrees $f_1$ and $f_2$ respectively, 
    then its periodic points have exact period $\leq p_1^{2f_1}p_2^{2f_2}$; and
    \item Zieve \cite{Zieve96}: if $\phi$ has good reduction at $\mfp$ above $p$ with residue degree $f$ and ramification degree $e$,
    then its periodic points have exact period
    $\leq O(e\rmn_{K/\Q}(\mfp)) = O(ep^f)$.
\end{itemize}

\subsection{Outline}
In Section \ref{sec:p-adic-fields}, we prove versions of
Theorems \ref{thm:uniform-boundedness} and \ref{thm:p-power-period}
for periodic points of polynomials in $p$-adic fields.
In Section \ref{sec:number-fields}, we use the natural embedding of number fields
into their $\mfp$-adic completions to deduce
the main results over number fields
from the statements about periodic points in $p$-adic fields.

\subsection*{Acknowledgments}

We are grateful to Rob Benedetto, John Doyle,
Xander Faber, Nicole Looper,
and Jit Wu Yap
for helpful discussions and suggestions
on this paper.
The second author is
supported by the National Science Foundation
under Grant No. DMS-2303280.


\numberwithin{theorem}{section}

\section{Periodic points in \texorpdfstring{$p$}{p}-adic fields}
\label{sec:p-adic-fields}

Throughout this section,
let $F$ be a finite extension of $\Q_p$
with residue degree $f$ and
ring of integers $\calo_F$
with unique maximal ideal $\mfp_F$,
fixed uniformizer $\varpi$,
and valuation $v_F$.

\subsection{Bounds on the number of periodic points}
\label{subsec:uniformbound}

\begin{figure}[ht]
    \centering
    \begin{tikzpicture}[scale=0.7, every node/.style={font=\small}]

\node[scale =1.5] at (-5.5,5.5) {$\mathbb{Q}_2(\sqrt{5})$};

\draw[very thick] (0,0) circle (6);
\fill[lightgray!40] (0,0) circle (6);
\node[scale=1.5] at (-3.5,3.5) { $\mathbb{Z}_2[\alpha]$};

\draw[thick] (0,3.5) circle (2.4);
\draw[thick] (0,-3.5) circle (2.4);
\draw[thick] (3.5,0) circle (2.4);
\draw[thick] (-3.5,0) circle (2.4);
\fill[red!15] (0,3.5) circle (2.4);
\fill[orange!15]  (0,-3.5) circle (2.4);
\fill[blue!15] (3.5,0) circle (2.4);
\fill[green!15] (-3.5,0) circle (2.4);

\draw (0,3.5+0.4*3.5) circle (2.4*0.4);
\draw (0,3.5-0.4*3.5) circle (2.4*0.4);
\draw (0+0.4*3.5,3.5) circle (2.4*0.4);
\draw (0-0.4*3.5,3.5) circle (2.4*0.4);
\fill[red!30] (0,3.5+0.4*3.5) circle (2.4*0.4);
\fill[orange!30] (0,3.5-0.4*3.5) circle (2.4*0.4);
\fill[blue!30]  (0+0.4*3.5,3.5) circle (2.4*0.4);
\fill[green!30] (0-0.4*3.5,3.5) circle (2.4*0.4);
\node[circle, fill, inner sep=1.5pt, label=below:{$\alpha + 2\alpha$}] at (0,3.5+0.4*3.5){};
\node[circle, fill, inner sep=1.5pt, label=below:{$\alpha + 2\beta$}] at (0,3.5-0.4*3.5){};
\node[circle, fill, inner sep=1.5pt, label=below:{$\alpha + 2$}] at (0+0.4*3.5,3.5){};
\node[circle, fill, inner sep=1.5pt, label=below:{$\alpha$}] (A) at (0-0.4*3.5,3.5){};
\draw (A) edge[loop above] node {} (A);

\draw[->,>=Latex,semithick, bend right=10] (0,3.5+0.4*3.5) to (0-0.4*3.5,3.5+0.6);
\draw[->,>=Latex,semithick, bend right=10] (0.4*3.5,3.5) to (0-0.4*3.5+0.6,3.5);
\draw[->,>=Latex,semithick, bend left=10] (0,3.5-0.4*3.5) to (0-0.4*3.5,3.5-0.7);

\draw (0,-3.5+0.4*3.5) circle (2.4*0.4);
\draw (0,-3.5-0.4*3.5) circle (2.4*0.4);
\draw (0+0.4*3.5,-3.5) circle (2.4*0.4);
\draw (0-0.4*3.5,-3.5) circle (2.4*0.4);
\fill[red!30] (0,-3.5+0.4*3.5) circle (2.4*0.4);
\fill[orange!30] (0,-3.5-0.4*3.5) circle (2.4*0.4);
\fill[blue!30] (0+0.4*3.5,-3.5) circle (2.4*0.4);
\fill[green!30] (0-0.4*3.5,-3.5) circle (2.4*0.4);
\node[circle, fill, inner sep=1.5pt, label=below:{$\beta + 2\alpha$}] at (0,-3.5+0.4*3.5){};
\node[circle, fill, inner sep=1.5pt, label=below:{$\beta + 2\beta$}] at (0,-3.5-0.4*3.5){};
\node[circle, fill, inner sep=1.5pt, label=below:{$\beta + 2$}] at (0+0.4*3.5,-3.5){};
\node[circle, fill, inner sep=1.5pt, label=below:{$\beta$}] (B) at (0-0.4*3.5,-3.5){};
\draw[->,>=Latex,semithick, bend right=10] (0,-3.5+0.4*3.5) to (0-0.4*3.5,-3.5+0.6);
\draw[->,>=Latex,semithick, bend right=10] (0.4*3.5,-3.5) to (0-0.4*3.5+0.6,-3.5);
\draw[->,>=Latex,semithick, bend left=10] (0,-3.5-0.4*3.5) to (0-0.4*3.5,-3.5-0.7);
\draw (B) edge[loop above] node {} (B);

\draw (3.5,0.4*3.5) circle (2.4*0.4);
\draw (3.5,-0.4*3.5) circle (2.4*0.4);
\draw (3.5+0.4*3.5,0) circle (2.4*0.4);
\draw (3.5-0.4*3.5,0) circle (2.4*0.4);
\fill[red!30] (3.5,0.4*3.5) circle (2.4*0.4);
\fill[orange!30] (3.5,-0.4*3.5) circle (2.4*0.4);
\fill[blue!30] (3.5+0.4*3.5,0) circle (2.4*0.4);
\fill[green!30] (3.5-0.4*3.5,0) circle (2.4*0.4);
\node[circle, fill, inner sep=1.5pt, label=below:{$-1 + 2\alpha$}](-12alpha) at (3.5,0.4*3.5){};
\node[circle, fill, inner sep=1.5pt, label=below:{$-1 + 2\beta$}](-12beta) at (3.5,-0.4*3.5){};
\node[circle, fill, inner sep=1.5pt, label=below:{$-1 + 2$}](-12) at (3.5+0.4*3.5,0){};
\node[circle, fill, inner sep=1.5pt, label=below:{$-1$}] (-1) at (3.5-0.4*3.5,0){};

\draw (-3.5,0.4*3.5) circle (2.4*0.4);
\draw (-3.5,-0.4*3.5) circle (2.4*0.4);
\draw (-3.5+0.4*3.5,0) circle (2.4*0.4);
\draw (-3.5-0.4*3.5,0) circle (2.4*0.4);
\fill[red!30]  (-3.5,0.4*3.5) circle (2.4*0.4);
\fill[orange!30] (-3.5,-0.4*3.5) circle (2.4*0.4);
\fill[blue!30]  (-3.5+0.4*3.5,0) circle (2.4*0.4);
\fill[green!30] (-3.5-0.4*3.5,0) circle (2.4*0.4);
\node[circle, fill, inner sep=1.5pt, label=below:{$2\alpha$}](2alpha) at (-3.5,0.4*3.5){};
\node[circle, fill, inner sep=1.5pt, label=below:{$2\beta$}](2beta) at (-3.5,-0.4*3.5){};
\node[circle, fill, inner sep=1.5pt, label=below:{$2$}](2) at (-3.5+0.4*3.5,0){};
\node[circle, fill, inner sep=1.5pt, label=below:{$0$}] (0) at (-3.5-0.4*3.5,0){};

\draw[->,>=Latex,semithick, bend right=20] (-3.5-0.4*3.5,0) to (-1);
\draw[->,>=Latex,semithick, bend left=34] (3.5-0.4*3.5,0) to (0);

\node (mergepoint) at (0,0) {};

\draw[->,>=Latex,semithick] (2) to[out=0,in=180] (mergepoint);
\draw[->,>=Latex,semithick] (2alpha) to[out=-10,in=180] (mergepoint);
\draw[->,>=Latex,semithick] (2beta) to[out=10,in=180] (mergepoint);
\draw[->,>=Latex,semithick] (-0.4,0) -- (3.5-0.4*3.5-0.6,0);

\node (mergepoint) at (3.5*0.44,3.5*0.27) {};
\draw[->,>=Latex,semithick] (-12) to[out=160,in=0] (mergepoint);
\draw[->,>=Latex,semithick] (-12alpha) to[out=192,in=0] (mergepoint);
\draw[->,>=Latex,semithick] (-12beta) to[out=90,in=0] (mergepoint);
\draw[->,>=Latex,semithick] ({3.5*0.52}, {3.5*0.27}) to[out=180,in=10] ({-3.5-0.4*3.5}, {0.5});

\draw[->,>=Latex,semithick] (6.5,0) to (7.5,0);
\draw[->,>=Latex,semithick] (-6.5,0) to (-7.5,0);
\draw[->,>=Latex,semithick] (0,6.5) to (0,7.5);
\draw[->,>=Latex,semithick] (0,-6.5) to (0,-7.5);

\end{tikzpicture}
    \caption{In $F = \mathbb{Q}_2(\sqrt{5})$, the map $\phi(X) = X^2-1$ has two fixed points, $\alpha = \frac{1+\sqrt{5}}{2}$ and $\beta = \frac{1-\sqrt{5}}{2}$, and two period-$2$ points, $0$ and $1$. Their nearby points are attracted, while faraway points are repelled. The largest circle represents $\calo_F = \Z_2[\alpha]$. The four medium-size circles represent open balls of radius $1$. The sixteen smaller circles represent open balls of radius $\frac{1}{2}$.}
    \label{figurediagram}
\end{figure}
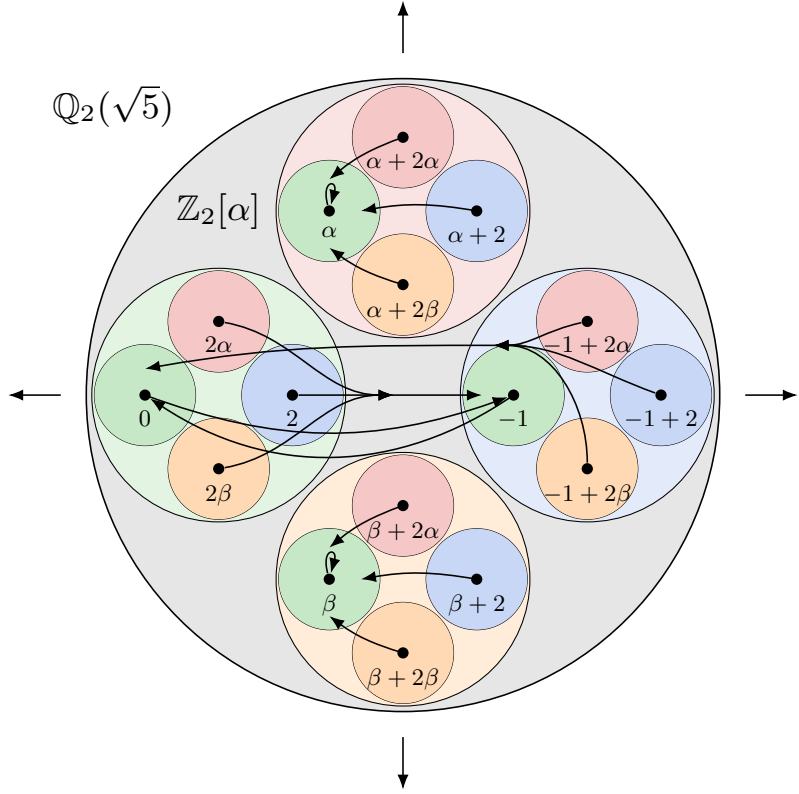
In this section, we prove a bound on the number of $p$-adic periodic points
of polynomials. This will then
be used to prove Theorem \ref{thm:uniform-boundedness} in Section \ref{subsec:numberfields}. To visualize the contraction and expansion argument used throughout this section, see Figure \ref{figurediagram}. Throughout this section, let $p \mid d$.
We consider polynomials of the form
\begin{itemize}[label=(\(\star\))]
    \item $\phi(X) = \displaystyle\sum_{i=0}^d a_iX^i \in \calo_F[X]\text{ such that } v_F(a_d) = 0 \text{ and } v_F(a_i) > 0 \text{ if } p \nmid i.$
\end{itemize}
\begin{proposition}
\label{prop:uniform-boundedness-p-adic}
    Let $\phi \in \calo_F[X]$ be a polynomial satisfying $(\star)$.
    Then $\phi$ has at most $p^f$ periodic points in $F$.
\end{proposition}

To prove Proposition \ref{prop:uniform-boundedness-p-adic}, we begin with an expansion argument that rules out periodic points $\lambda$ of $\phi$ with $\verts{\lambda}_F > 1$.
\begin{lemma}
    \label{lem:expansion}
    Let $\phi$ be as in ($\star$). If $\lambda \in F$ with $\verts{\lambda}_F > 1$, then $\lambda$ is not a periodic point of $\phi$.
\end{lemma}
\begin{proof}
    Since $\verts{\lambda}_F > 1$,
    we have $\verts{\lambda^d}_F > \verts{\lambda^{d-1}}_F
    > \cdots > \verts{\lambda}_F > 1$.
    By $(\star)$, we have $\verts{a_i}_F \leq 1$ for all $i$,
    and $\verts{a_d}_F = 1$.
    So, $\verts{a_d \lambda^d}_F > \verts{a_i \lambda^i}_F$
    for all $i < d$. This gives
    \[
        \verts{\phi(\lambda)}_F
            = \verts{\sum_{i=0}^d a_i \lambda^i}_F 
            = \max_i \Big(\verts{a_i \lambda^i}_F\Big)
            = \verts{a_d \lambda^d}_F > \verts{\lambda}_F.
    \]
    Hence, $\verts{\lambda}_F < \verts{\phi(\lambda)}_F
    < \verts{\phi^2(\lambda)}_F < \verts{\phi^3(\lambda)}_F
    < \cdots\;,$
    so $\lambda$ is not periodic.
\end{proof}

Next, we give a contraction argument that says that
periodic points of $\phi$ cannot be too close together,
since each periodic point will attract points near it.
This attraction occurs because $\verts{\phi'(\lambda)}_F < 1$ for any periodic $\lambda$, by ($\star$) and Lemma \ref{lem:expansion}.

\begin{lemma}
    \label{lem:contraction}
    Let $\phi$ be as in $(\star)$.
    If $\lambda \in F$ is a periodic point of $\phi$
    and $x \in F - \{\lambda\}$ such that $\verts{x-\lambda}_F < 1$,
    then $x$ is not a periodic point of $\phi$.
\end{lemma}
\begin{proof}
    Let $\lambda$ have exact period $r$. First, assume $x$ is periodic and assume there is some $s > 0$ such that $\phi^s(x) = \phi^s(\lambda) =: \nu$. Then, $\nu$ has exact period $r$, so $x$ has exact period $r$ too. Then $\lambda = \phi^{r-s}(\nu) =x$, which is a contradiction. So, we can assume that $\phi^s(x) \neq \phi^s(\lambda)$ for all $s \geq 1$. 

    Define the $0$-th iterate $\phi^{0}(a) = a$. We show by induction on $n$ that for all $n \geq 0$, then \[ 0 <\verts{ \phi^n(x)-\phi^n(\lambda)}_F < \verts{\phi^{n-1}(x) - \phi^{n-1}(\lambda)}_F< \cdots < \verts{\phi(x) - \phi(\lambda)}_F < \verts{x-\lambda}_F < 1.\] With $n=0$ this follows from our assumptions. Now, let $m\geq 1$ and assume this hypothesis is true for $n = m-1$. We will show it holds for $n = m$. Let $\mu = \phi^{m-1}(\lambda)$ and $y = \phi^{m-1}(x)$. Using a Taylor expansion around $\mu$, with $\phi^{(k)}$ denoting the $k$-th derivative of $\phi$, we have \[
        \phi^{m}(x) = \phi(y) = \sum_{k=0}^d \frac{1}{k!} \phi^{(k)}(\mu)(y-\mu)^k = \sum_{k=0}^d\sum_{i=k}^d a_i \binom{i}{k} \mu^{i-k}(y-\mu)^k.
    \]
    In this expansion, the $k=0$ term is simply $\phi(\mu) = \phi^m(\lambda)$, so we can subtract it from both sides. Notice that $\verts{\mu}_F \leq 1$ since $\mu$ is periodic, by Lemma~\ref{lem:expansion}, and $0 < \verts{y-\mu}_F <1$ by the inductive hypothesis. Also, for all $i$, we have $\verts{a_ii}_F  <1$ and $\verts{a_i\binom{i}{k}}_F \leq 1$ by the conditions of $(\star)$. Using all of these facts gives 
\[
\verts{\phi^m(x)-\phi^m(\lambda)}_F \leq \max_{\substack{1 \leq k \leq d \\ k \leq i \leq d}}\verts{a_i \binom{i}{k} \mu^{i-k}(y-\mu)^k}_F< \verts{y-\mu}_F.
\]
Therefore, $0 < \verts{\phi^m(x)-\phi^m(\lambda)}_F < \verts{\phi^{m-1}(x)-\phi^{m-1}(\lambda)}_F < 1$, so our inductive hypothesis holds. Therefore, $x$ is not periodic since its iterates approach the iterates of $\lambda$ monotonically.
\end{proof}

\begin{remark}
    We are grateful to Nicole Looper for pointing out to us
    that Feiner \cite[Proposition 10]{feiner} provides
    a backwards-iteration analogue of Lemma \ref{lem:contraction}
    in a more general Berkovich framework.
    Adapting Feiner's geometric viewpoint to our
    forwards-iteration result appears to be a natural route to broader generalizations.
\end{remark}

Altogether, the contraction and expansion arguments
give a separation of periodic points into distinct residue classes.
This allows the proof of Proposition \ref{prop:uniform-boundedness-p-adic}, which will be the main tool in the proof of Theorem \ref{thm:uniform-boundedness}.

\begin{proof}[Proof of Proposition \ref{prop:uniform-boundedness-p-adic}]
    By Lemma \ref{lem:expansion}, all periodic points $\lambda$ of $\phi$
    satisfy $\verts{\lambda}_F \leq 1$, meaning they are in $\calo_F$. Since $F$ has residue degree $f$ as an extension of $\Q_p$, we have
    $\rmn_{F/\Q_p}(\mfp) = p^f$ and
    $\calo_F/\mfp_F \simeq \F_{p^f}$ (cf. \cite[Theorem 6.4.6]{gouvea}).
    By Lemma~\ref{lem:contraction},
    all periodic points of $\phi$ in $\calo_F$ 
    must be in distinct residue classes of $\calo_F/\mfp_F$.
    Hence there are at most $p^f$
    many periodic points of $\phi$.
\end{proof}

\subsection{Bounds on exact periods in prime power degree}
We now establish bounds on the possible periods for $p$-adic
periodic points of prime-power-degree polynomials.
This will then be used to prove Theorem \ref{thm:p-power-period}
in Section \ref{subsec:numberfields}. Throughout this section, let $d = p^k$ for some positive integer $k$.
We consider polynomials of the form:
\begin{itemize}[label=(\(\star\star\))]
    \item $\phi(X) = \displaystyle\sum_{i=0}^d a_i X^i \in \calo_F[X] \text{ such that } a_d = 1 \text{ and } v_F(a_i) > 0 \text { if }i \neq 0,d.$
\end{itemize}
Note that these polynomials are all of the form $(\star)$ from Section \ref{subsec:uniformbound}.

Define the dynatomic polynomial of $\phi$:\footnote{
It is nontrivial to see that this formula defines a polynomial;
see \cite[Section 4.1]{Silverman07}.
Here, $\mu$ is the M\"obius function, which is defined by 
$\mu(j) = 0$ if $j$ is divisible by the square of some prime, and 
$\mu(j) = (-1)^k$ if $j$ is the product of $k$ distinct primes.}
\begin{equation}\label{defdynatomic}
    \Phi_{n}(X) \coloneq \prod_{i \mid n} (\phi^i(X)-X)^{\mu(n/i)}\;.
\end{equation}

As in Section \ref{sec:intro},
define $m(a,b,c,d) \in \N$ for $a,b,c,d \in \calo_F$
to be the smallest positive integer $m$
such that $a \mid cm$ and $b \mid dm$.

\begin{proposition}
\label{prop:p-power-p-adic}
    Let $\phi \in \calo_F[X]$ be a polynomial satisfying $(\star \star)$.
    Assume that $f \mid k$ or $a_0 \in \Q$.
    \begin{enumerate}[label = (\alph*)]
        \item There are exactly $p^f$
            many periodic points of $\phi$ in $F$,
            all with exact period dividing $m = m(\varpi,f,a_0,k)$.
        \item If $km = f$, then for each $n \mid m$,
            there are exactly
            $\deg_X(\Phi_{n}(X))$ many periodic points of
            $\phi$ of exact period $n$ in $F$.
    \end{enumerate}
\end{proposition}

\begin{nonexample}
    The assumption of Proposition \ref{prop:p-power-p-adic}
    that $f \mid k$ or $a_0 \in \Q$
    is necessary in the sense that
    the conclusion of Proposition~\ref{prop:p-power-p-adic}(a) is not true
    if $a_0 \in \calo_F \setminus \Q$ and $f \nmid k$.
    For example, let $F = \Q_2(\sqrt{-3})$, which has $f = 2$ and has uniformizer $\varpi = 2$. Let $d = 2$, so $k = 1$. Let $\omega = -\frac{1}{2}+\frac{\sqrt{-3}}{2}$. Let $\phi = X^2+\omega$. Let $g(X) = \phi^4(X)-X$. Then, one can check that $g(\omega) = -2-2\sqrt{-3}$ and $g'(\omega) \equiv 1 \pmod{2}$. From Hensel's Lemma (cf. \cite[6.5.2]{gouvea}), there is $\alpha \in F$ with $\alpha \equiv \omega \pmod{2}$ and $g(\alpha) = 0$. We can check manually that $\phi^i(\omega) \not \equiv \omega \pmod{2}$ for $1 \leq i \leq 3$. Therefore, $\alpha \in F$ is a periodic point of exact period $4$, which does not divide $m = 2$.
\end{nonexample}

To prove Proposition \ref{prop:p-power-p-adic},
we first describe the iterates of $\phi$
modulo the uniformizer $\varpi$, with a lemma proved in a similar manner to \cite[Lemma 2.3]{HS24}.

    \begin{lemma}\label{lemma2.5}
        Let $\phi$ be as in $(\star \star)$, and assume that $f \mid k$ or $a_0 \in \Q$. Letting $n \in \N$, then 
        \[
            \phi^{n}(X) \equiv X^{d^{n}} + a_0n \pmod \varpi.
        \]
    \end{lemma}
    
    \begin{proof}
        Equality of two elements in $\calo_F \pmod{\varpi}$ means they are mapped to the same element in $\calo_F/\mfp_F$. 
        Reducing modulo $\varpi$, we work in $ \calo_F/\mfp_F \simeq F_{p^f}$.
    Since $d = p^k$, the map $x \mapsto x^d = x^{p^k}$
    is the $k$-th iterate of the Frobenius automorphism on $\F_{p^f}$.
    In particular, we have $x^{p^f} \equiv x \pmod \varpi$
    and
        $(x + y)^d \equiv x^d + y^d \pmod \varpi$ 
    for all $x, y \in \calo_F / \mfp_F$.
        We induct on $n$. Let $i \in \N$ and
        assume that
        \[
            \phi^i(X) \equiv X^{d^i} + a_0i \pmod \varpi \;.
        \]
        We will show the lemma to hold for $i+1$.
        If $f \mid k$, then $(a_0i)^d \equiv a_0i \pmod \varpi$. If $a_0 \in \Q$, then $a_0i$ is in $\Q$ and thus $a_0i \equiv s \pmod{\varpi}$ for some $s \in \{0,1,\ldots,p-1\}$. In that case, the Frobenius automorphism fixes $a_0i \pmod{\varpi}$, so again $(a_0i)^d \equiv a_0i \pmod \varpi$.
        Hence, 
        \begin{alignat*}{3}
            \phi^{i+1}(X)
                &\equiv \paren{X^{d^{i}} + a_0i}^d + a_0  &&\pmod \varpi \\
                &\equiv X^{d^{i+1}} + (a_0i)^d + a_0    &&\pmod \varpi \\
                &\equiv X^{d^{i+1}} + a_0(i + 1)         &&\pmod \varpi,
        \end{alignat*}
        so we are done.
    \end{proof}

\begin{proof}[Proof of Proposition \ref{prop:p-power-p-adic}]
    To prove (a), it suffices to find $p^f$ periodic points of $\phi$ in $F$ of period $m$. Then, Proposition \ref{prop:uniform-boundedness-p-adic} tells us that these are all of them. Points of period $m$ are given by roots of $\phi^m(X) -X$.
    
    By the definition of $m$, we have $\varpi \mid a_0m$, so $a_0m \equiv 0 \pmod{\varpi}$. Then, letting $n = m$ and plugging into Lemma \ref{lemma2.5} yields
    \[
        \phi^m(X) \equiv X^{d^m} + a_0m \equiv X^{p^{km}} \pmod \varpi.
    \]
    
    By the definition of $m$, we also have that $f \mid km$. Hence, $p^{f}-1 \mid p^{km} -1$. 
    Since $|(\calo_F/\mfp_F)^\times| = p^{f}-1$, Lagrange's theorem gives that all $p^f-1$ elements of $(\calo_F/\mfp_F)^\times$ are roots of $X^{p^{km} -1}-1$. So, all $p^f$ elements of $\calo_F/\mfp_F$ are roots of $X^{p^{km}}-X$.
    
    Let $g(X) = \phi^m(X)-X$, and let $\alpha \in \calo_F/\mfp_F$ be any of these $p^f$ roots of $X^{p^{km}}-X$. Then, $\verts{g(\alpha)}_F \equiv 0 \pmod{\varpi}$ and $g'(\alpha) = -1\pmod{\varpi}$. By Hensel's Lemma (cf. \cite[Theorem 6.5.2]{gouvea}), there exist exactly $p^f$ unique roots of $g(X)$ in $\calo_K$. These are the $p^f$ period $m$ points of $\phi$.

   We now prove (b). Morton and Patel \cite[Theorem 2.4(c)]{morton-patel} showed that if $\Phi_{n}(X)$ has no double roots in $X$,
    then the roots of $\Phi_{n}(X)$ are precisely the set of points of exact period $n$ of $\phi$. M\"obius inversion on (\ref{defdynatomic}) yields the following formula: 
    \begin{equation}\label{eqMobius}
    \phi^m(X)-X = \prod_{n \mid m} \Phi_{n}(X)\;.
    \end{equation}    
    Now, let $km = f$. Then, the left side has $p^f$ unique roots in $F$, and has degree $p^f$. So, $\Phi_n(X)$ has $\deg_X( \Phi_n(X))$ many unique roots in $F$ for all $n \mid m$. Therefore, there are exactly $\deg_X (\Phi_n(X))$ many periodic points of $\phi$ of exact period $n$.
\end{proof}

\subsection{Periodic points of quadratic polynomials}
In this section, we specialize to $p=2$
and use the ramification theory at the prime $2$ to
more precisely classify the periodic points of quadratic polynomials $\phi_{2,c}$
in extensions $F$ of $\Q_2$.
This extends the results in
\cite[Theorems 7 and 8]{walde-russo}
from $\Q_2$ to any finite extension of $\Q_2$
and will be used to classify
periodic points of quadratic polynomials
in quadratic number fields.
\begin{proposition}
    \label{prop:quadratic-2-adic}
    Let $F$ be a finite extension of $\Q_2$ with residue degree $f$.
    Let $c = r/s \in \Q$ with $s$ odd, and let $\phi_{2,c}$ be the polynomial $X^2+c$. Then:
    \begin{enumerate}[label = (\alph*)]
        \item If $rf$ is even, then all periodic points of $\phi_{2,c}$ in $F$ have period $f$, and for each $n \mid f$, there are exactly $\deg_X(\Phi_{n}(X))$ many periodic points of $\phi_{2,c}$ in $F$ of exact period $n$.
        \item If $rf$ is odd, then all periodic points of $\phi_{2,c}$  in $F$ have period $2f$.
    \end{enumerate}
\end{proposition}
\begin{proof}[Proof of Proposition \ref{prop:quadratic-2-adic}]
    By assumption, $d = p = 2$ and $k = 1$ and $a_0 = c = r/s$.
 The smallest positive integer $m$ in such that
    $\varpi \mid a_0m$ and $f \mid m$ in $\calo_F$ is
    \[
        m\left(\varpi, f, \frac{r}{s}, 1\right) =
            \begin{cases}
                f & \text{if } 2 \mid rf, \\
                2f & \text{if } 2 \nmid rf.
            \end{cases}
    \]
    Applying Proposition \ref{prop:p-power-p-adic}
    gives all of the desired results.
\end{proof}


\section{From \texorpdfstring{$p$}{p}-adic fields to number fields}
\label{sec:number-fields}

\subsection{General number fields}\label{subsec:numberfields}

In this section, we
use the natural embeddings of number fields in $p$-adic fields
and Propositions \ref{prop:uniform-boundedness-p-adic},
\ref{prop:p-power-p-adic},
and \ref{prop:quadratic-2-adic}
to prove Theorem~\ref{thm:uniform-boundedness},
Theorem \ref{thm:p-power-period}, and Corollary \ref{cor:2-general}
respectively.

Let $\mfp$ be a prime of a number field $K$
above the rational prime $p$.
Let $f = [\calo_K/\mfp : \F_p]$ denote its residue field degree.
Let $K_\mfp$ be the completion of $K$ with respect to $\mfp$. 
The natural embedding $K \hookrightarrow K_\mfp$
makes $K_\mfp/\Q_p$ a finite extension with
residue field degree $f \coloneq  [\calo_{K_\mfp} / \mfm_{K_\mfp} : \F_p]$,
where $\mfm_{K_\mfp}$ is the maximal ideal of
the valuation ring $\calo_{K_\mfp}$ of $K_\mfp$.

\begin{proof}[Proof of Theorem \ref{thm:uniform-boundedness}]
    Let $F = K_\mfp$, the $\mfp$-adic completion of $K$.
    Using the embedding $K \hookrightarrow F$,
    the periodic points of $\phi$ in $K$
    are a subset of periodic points in $F$.
    Applying Proposition \ref{prop:uniform-boundedness-p-adic}
    yields the upper bound
    $\#\Per_K(\phi) \leq \rmn_{K/\Q}(\mfp) = p^f$.
    Furthermore, $p^f \leq d^D$
    since the residue degree $f$ of $\mfp$ over $p$
    is less than $D = [K:\Q]$
    and since $p \mid d$.
\end{proof}

\begin{proof}[Proof of Theorem \ref{thm:p-power-period}]
    Under the embedding $K \hookrightarrow K_\mfp$,
    the periodic points over $K$ map to periodic points over $K_\mfp$.
    Applying Proposition \ref{prop:p-power-p-adic} with $F = K_\mfp$
    yields Theorem \ref{thm:p-power-period}. 
\end{proof}

Specializing Theorems \ref{thm:uniform-boundedness} and \ref{thm:p-power-period}
to $d=2$ and $p=2$ yields bounds of $2^f$ and $m(\mfp, f, a_0, 1) \in \{f, 2f\}$.
In this special case, additional ramification and degree considerations
yield the following slight refinement
of Theorems \ref{thm:uniform-boundedness} and \ref{thm:p-power-period}.

\begin{corollary}
    \label{cor:2-general}
    Let $K$ be a number field with ring of integers $\calo_K$,
    let $\mfp$ be a prime of $\calo_K$ above $2$ with 
    residue degree $f$,
    and let $v_\mfp(c) \geq 0$.
    \begin{enumerate}[label = (\alph*)]
        \item If $2$ splits completely or is totally ramified in $K$,
            then $\Per_K(\phi_{2,c})$ consists of zero or two fixed points
            (resp. periodic points of exact period $2$)
            when $v_\mfp(c) >0$
            (resp. $v_\mfp(c)=0$).
        \item If $c = r/s \in \Q$, then
            $\#\Per_K(\phi_{2,c}) \leq 2^f$.
            Furthermore, each $x \in \Per_K(\phi_{2, c})$
            has period $f$ (resp. $2f$) when
            $rf$ is even (resp. $rf$ is odd).
    \end{enumerate}
\end{corollary}

\begin{proof}
    Observe that the setting when $2$ splits completely
    or is totally ramified in $K$
    corresponds to residue degree
    $f = 1$.
    By Theorem \ref{thm:uniform-boundedness},
    the quadratic polynomial $\phi_{2,c}$ has at most two periodic points. Then, notice that $\calo_{K,\mfp}/\mfp \calo_{K,\mfp} \simeq \calo_K/\mfp \calo_K \simeq \F_2$. 
    By \cite[Theorem 1]{walde-russo},
    the fixed points of the quadratic polynomial $\phi_{2,c}$ are the roots of $g(X) \coloneq  X^2 - X + c$ and the periodic points of exact period $2$ are the roots of $h(X) \coloneq  X^2 + X + 1 + c$ when $c \neq -\frac{3}{4}$.
    Consider reductions of these polynomials to $\F_2$. If $v_\mfp(c) > 0$, then $\overline{g}(X)$ has no repeated roots in $\F_2$ and $\overline{h}(X) = X^2+X+1$ has no roots in $\F_2$. Therefore, $\phi_{2,c}$ has either zero or two fixed points,
    and has no other periodic points.
    If $v_\mfp(c) = 0$, then $\overline{g}(X) = X^2+X+1$ has no roots in $\F_2$, so $\phi_{2,c}$ has zero or two points of exact period $2$, and no other periodic points.
    These two cases yield Corollary \ref{cor:2-general}(a).

    Again consider the embedding $K \hookrightarrow K_\mfp$. 
    Applying Propositions \ref{prop:p-power-p-adic} and \ref{prop:quadratic-2-adic}
    with $F = K_\mfp$ immediately
    gives Corollary \ref{cor:2-general}(b).
\end{proof}

\begin{remark}
    \label{rem:walde-russo}
    The totally split case of
    Corollary \ref{cor:2-general}(a)
    can be proved as a direct consequence of
    theorems by Walde--Russo
    \cite[Theorems 7 and 8]{walde-russo}
    about periodic points of $\phi_{2, c}$ in $\Q_2$,
    using the fact that
    a totally split prime $\mfp$ over $2$
    gives rise to an embedding
    $K \hookrightarrow \Q_2$. 
\end{remark}

\subsection{Quadratic number fields}

Finally, we specialize further to
quadratic polynomials.
The dynamics of quadratic polynomials can be reduced to the
dynamics of $\phi_{2, c}(X) = X^2 + c$ by linear conjugation.
For quadratic number fields $K = \Q(\sqrt{\Delta})$,
it is conjectured (cf. \cite{doyle-faber-krumm,HI13})
that all $K$-rational periodic points of $\phi_{2, c} \in \Q[X]$
have exact period $n<5$,
with a single exception given by the $6$-cycle of $X^2 - \frac{71}{48}$
in $\Q(\sqrt{33})$;
the second author \cite[Corollaries 1.8 and 1.9]{Zhang21}
proved this when $c$ is not
in a finite set $\Sigma_{2,n} \subset \Q$
for $n = 5$ and conditionally for $n=6$. 
We apply Theorem \ref{thm:p-power-period}
to give a classification of periodic points of
$\phi_{2,c}$ over quadratic number fields
when $c \in K$ is integral at $\mfp$.
\begin{corollary}
    \label{cor:2-quadratic-field}
    Let $K = \Q(\sqrt{\Delta})$ be a quadratic number field,
    $\mfp$ be a prime of $\calo_K$ above 2, and $c \in \calo_{K,\mfp}$.
    \begin{itemize}
        \item If $\Delta \not\equiv 5 \pmod 8$,
            then $\#\Per_K(\phi_{2, c}) = 0$ or $2$.\\
            Furthermore, each $x \in \Per_K(\phi_{2, c})$
            is a fixed point (resp. point of exact period $2$)
            when $v_\mfp(c) > 0$
            (resp. $v_\mfp(c) = 0$).
        \item If $\Delta \equiv 5 \pmod 8$,
            then $\#\Per_K(\phi_{2, c}) = 0$, $2$, or $4$.\\
            Furthermore, each $x \in \Per_K(\phi_{2, c})$
            has exact period $\leq 2$ when $c \in \Q$.
    \end{itemize}
\end{corollary}

\begin{proof}[Proof of Corollary \ref{cor:2-quadratic-field}]
Fix the non-archimedean absolute value
$\lvert \, \cdot \, \rvert_{\mfp}$ on $K$
normalized by
\[
    \lvert x \rvert_\mfp
        = \rmn(\mfp)^{-v_{\mfp}(x)}
        = (2^f)^{-v_{\mfp}(x)},
\]
so that $\lvert x \rvert_\mfp > 1$
if and only if $v_\mfp(x) < 0$.

    The splitting of $2$ in a
    quadratic number field $K = \Q(\sqrt{\Delta})$
    is precisely given by:
    \begin{itemize}
        \item if $\Delta \equiv 1 \pmod 8$,
            then $2$ splits completely and $f = 1$;
        \item if $\Delta \equiv 2, 3 \pmod 4$,
            then $2$ is totally ramified and $f = 1$;
        \item if $\Delta \equiv 5 \pmod 8$,
            then $2$ is inert and $f = 2$.
    \end{itemize}
    It follows directly from Corollary \ref{cor:2-general}(a)
    that if $\Delta \not\equiv 5 \pmod 8$, then
    $\#\Per_K(\phi_{2, c}) \in \{0, 2\}$
    and each $x \in \Per_K(\phi_{2, c})$
    is a fixed point (resp. point of exact period $2$)
    when $v_\mfp(c) > 0$
    (resp. $v_\mfp(c) = 0$).
    
    Suppose that $\Delta \equiv 5 \pmod{8}$.
    By Theorem \ref{thm:uniform-boundedness},
    $\#\Per_{K}(\phi_{2,c}) \leq 2^2 = 4$.
    By the same argument as in
    the proof of Corollary \ref{cor:2-general},
    $\phi_{2,c}$ cannot have exactly $1$ fixed point.
    Therefore $\#\Per_{K}(\phi_{2,c}) \in \{0,2,3,4\}$
    when $\Delta \equiv 5 \pmod{8}$.

    We prove by contradiction that
    $\#\Per_{K}(\phi_{2,c}) \neq 3$
    in the $\Delta \equiv 5 \pmod 8$ case.
    Suppose that $\Delta \equiv 5 \pmod{8}$ and
    $\#\Per_K(\phi_{2, c}) = 3$.
    Since $\phi_{2, c}$ cannot have exactly $1$ or $3$ fixed points,
    there must exist a periodic point of $\phi_{2, c}$
    of exact period $3$.
    By \cite[Theorem 3]{walde-russo},
    the quadratic polynomial $\phi_{2,c}$
    has a $K$-rational $3$-cycle if and only if
    \[
         c = -\frac{\tau^6+2\tau^5+4\tau^4+8\tau^3+9\tau^2+4\tau+1}{4\tau^2(\tau+1)^2}
    \]
    for some $\tau \in K$.
    If $\lvert \tau \rvert_\mfp > 1$ then
    $\lvert c \rvert_\mfp = \verts{\frac{\tau^6}{4\tau^4}}_\mfp > 1$.
    If $\lvert \tau \rvert_\mfp \leq 1$,
    then $\tau \in \calo_{K,2}$.
    Let $\overline{\tau}$ be the reduction of $\tau$ in $\calo_{K,2}/2\calo_{K,2} \cong \calo_{K}/2\calo_{K} \cong \F_4$; 
    observe that the numerator $\overline{\tau}^6+2\overline{\tau}^5+4\overline{\tau}^4+8\overline{\tau}^3+9\overline{\tau}^2+4\overline{\tau}+1$
    is always nonzero in $\F_4$.
    Thus, the numerator has absolute value $1$ and
    $\lvert c \rvert_\mfp = \frac{1}{\verts{4\tau^2(\tau+1)^2}_\mfp} > 1$.
    But this contradicts the hypothesis that
    $c \in \calo_{K, \mfp}$.
    Therefore, no $K$-rational $3$-cycle of $\phi_{2, c}$
    can exist, and
    $\#\Per_{K}(\phi_{2,c}) \in \{0,2,4\}$
    when $\Delta \equiv 5 \pmod{8}$.

    Finally, the assertion that each $x \in \Per_K(\phi_{2,c})$ has exact period $\leq 2$ when $c \in \Q$ follows directly from Corollary~\ref{cor:2-general}(b).
\end{proof}
\begin{remark}
    John Doyle has kindly noted to us that
    the conclusions of Corollary~\ref{cor:2-general}(a)
    and part of Corollary~\ref{cor:2-quadratic-field}
    can be derived from \cite[Theorem~2.21]{Silverman07}.
    We include our arguments here because, within our approach,
    these results emerge as direct and
    particularly simple consequences of the main theorems.
\end{remark}

We highlight how the number of periodic points of $\phi_{2,c}$ 
in Corollary \ref{cor:2-quadratic-field}
can vary through
the following example, in which there are
$4$ periodic points over the quadratic extension $F/\Q_2$
but $2$ or $4$ periodic points over the quadratic extension $K/\Q$.

\begin{example}
    Let $K = \Q(\sqrt{5})$, $K' = \Q(\sqrt{-3})$, and $F = K_2 = K'_2 = \Q_2(\sqrt{5})$.
    Over $F$, the quadratic polynomial $\phi \coloneq X^2 - 1$ has
    two fixed points $\{\frac{1}{2} \pm \frac{\sqrt{5}}{2}\}$
    and two periodic points $\set{-1, 0}$ of exact period $2$, as shown in Figure \ref{figurediagram}.
    These are also the four periodic points of $\phi$ over $K$,
    but $\set{-1, 0}$ are the only two periodic points of $\phi'$ over $K'$.
    Note that Corollary \ref{cor:2-quadratic-field}
    does not distinguish between $K$ and $K'$
    since $2$ is inert in both fields.
\end{example}

An application of Corollary \ref{cor:2-quadratic-field}
verifies the conjecture of Doyle--Faber--Krumm
\cite[Speculation 1.4]{doyle-faber-krumm}
and Doyle \cite[Conjecture 1.4]{doyle-2018}
for the field $\Q(i)$ 
and the field $\Q(\zeta_3) = \Q(\sqrt{-3})$ 
when $c$ is integral
at the unique prime above $2$.
The bound ${\#\Per_K(\phi_{2, c}) \leq 4}$ from
Corollary~\ref{cor:2-quadratic-field}
and a result of Doyle \cite[Theorem 1.4]{doyle-2020}
together imply
that $\# \Preper_K(\phi_{2,c}) \leq 10$.
Furthermore,
Corollary~\ref{cor:2-quadratic-field} 
strengthens the conclusions of
\cite[Theorem 1.4]{doyle-2020}
on the possible portraits $\calp_{K, \phi_{2, c}}$
of preperiodic points of $\phi_{2, c}$ to:
\begin{itemize}
    \item if $c \in \Z[i]_{(1+i)}$, then
        \[
            \calp_{\Q(i), \phi_{2, c}} \in \set{0, 3(2), 4(1,1), 4(2), 5(1,1)a/b, 5(2)a, 6(1,1), 6(2)};
        \]
   \item if $c \in \Z[\zeta_3]_{(2)}$, then
        \[
            \calp_{\Q(\zeta_3), \phi_{2, c}} \in \set{0, 3(2), 4(1,1), 4(2), 5(1,1)a, 6(1,1), 6(2), 7(2,1,1)a, 8(2)a, 8(2,1,1)}.
        \]
\end{itemize}
Here, the portraits are labeled
by the number of preperiodic points and
the tuple of sizes of cycles
as in \cite[Appendix B]{doyle-2020}.


\bibliography{bibliography}{}
\bibliographystyle{plain}

\end{document}